\documentclass[letterpaper, 10 pt, conference]{ieeeconf}  

\IEEEoverridecommandlockouts                              

\usepackage{amsmath}
\usepackage{comment}
\usepackage{graphicx}
\usepackage{amsfonts}
\usepackage{booktabs}
\usepackage{multirow}
\usepackage{tabularx}
\usepackage{array}
\usepackage{makecell}
\usepackage{xcolor}  
\usepackage{url}
\usepackage{hyperref}
\newtheorem{theorem}{Theorem}
\newtheorem{lemma}{Lemma}

\newtheorem{assumption}{Assumption}
\newtheorem{remark}{Remark}
\newtheorem{definition}{Definition}
\newtheorem{corollary}{Corollary}

\title{Provably Safe Decentralized Contingency MPC under State-Only Information and Limited Sensing for Nonlinear Multi-agent Systems}

\author{Max Studt$^1$, Georg Schildbach$^1$}

\begin{document}
\maketitle

{
  \renewcommand{\thefootnote}{}%
  \footnotetext[1]{$^1$Institute for Electrical Engineering in Medicine of the University of Luebeck, Germany, (\emph{\{m.studt, georg.schildbach\}@uni-luebeck.de})}
}

\begin{abstract}
This paper considers decentralized contingency MPC for multi-agent control under a state-only information pattern, with particular focus on limited sensing and plug-and-play operation. The objective is to retain recursive feasibility, safety, and Lyapunov-type convergence while reducing conservatism in local interaction handling. The framework relies on agent-wise fallback regions (safe sets) in which a feasible contingency maneuver to a safe equilibrium is always available. A novel safe-set update mechanism is introduced that supports less conservative decentralized interaction while preserving the underlying guarantees. This, in turn, enables memory-free local interaction and finite sensing ranges without requiring agents to reconstruct the exact neighbor geometry. The resulting scheme remains fully decentralized and preserves the shared-first-input contingency MPC structure. Theoretical guarantees and simulation results illustrate the effectiveness of the approach in dense multi-agent scenarios.
\end{abstract}
 
\section{Introduction}

Decentralized collision avoidance without inter-agent communication is a challenging problem in multi-agent control. In this setting, each agent has to plan its own motion based only on local state measurements of nearby agents, while future trajectories, inputs, and intentions of other agents are unavailable. Such a state-only information pattern is particularly relevant when communication is unreliable, unavailable, or undesirable, and when agents can only rely on finite-range sensing.

Reactive geometric methods, such as velocity obstacles and reciprocal collision avoidance, address collision avoidance under limited information and without trajectory exchange \cite{FioriniShiller1998,vanDenBergORCA2011}. These approaches are attractive due to their decentralized nature and low computational complexity. However, they are typically not formulated as predictive controllers with explicit multi-step state and input constraints, recursive feasibility, and closed-loop convergence guarantees. Predictive approaches can handle such constraints and optimize performance over a horizon, but decentralized MPC schemes often require some form of communication, trajectory exchange, or coordination \cite{DUNBAR2006549,Saccani2023,CHRISTOFIDES201321}. Communication-free predictive control schemes based on local sensing, neighbor-trajectory prediction, or iterative learning have also been proposed \cite{NeigborPrediction,Borelli}. However, these approaches rely on predicted neighbor behavior or data from previous task executions, or do not jointly provide MPC-style closed-loop guarantees under state-only information. This leaves a gap between communicationless local interaction
and predictive control with rigorous closed-loop guarantees.

Contingency MPC is particularly useful in this setting because it separates performance optimization from safety certification \cite{Alsterda2019}. Each agent computes a nominal trajectory for its control objective and, in parallel, a contingency trajectory that certifies the existence of a feasible fallback maneuver. In the context of communicationless decentralized collision avoidance, contingency trajectories together with separating bisecting planes have been used to establish recursive feasibility and closed-loop collision avoidance \cite{Georg}. However, this approach is formulated only for a linear double-integrator robot-swarm model and does not provide a general framework for nonlinear agent dynamics or a Lyapunov-type convergence result.

A contingency MPC approach based on local safe sets provides a complementary framework \cite{StudtSchildbach2026FoS}. However, in this framework, agents require information about other agents from previous time steps, introducing a history dependence that complicates plug-and-play operation and limited sensing.

This paper builds on \cite{StudtSchildbach2026FoS} and removes the need for history-dependent neighbor reconstruction by introducing a memory-free, state-dependent interaction mechanism. The active safe set is constructed such that it always remains contained in a region computable from the agent's current state. The resulting scheme preserves the contingency-MPC guarantees of recursive feasibility, collision avoidance, and Lyapunov-type convergence, while replacing history-dependent neighbor reconstruction by a memory-free, state-dependent interaction mechanism. This makes the approach naturally compatible with finite sensing and continuous plug-and-play operations.

The main contributions of this paper are as follows:
\begin{itemize}
    \item[C1] A memory-free, state-dependent interaction mechanism for decentralized contingency MPC.

    \item[C2] A safe-set construction with a memory-free outer representation recoverable from current state information.

    \item[C3] A finite-sensing and plug-and-play formulation with reduced conservatism.

    \item[C4] Recursive-feasibility, collision-avoidance, and Lyapunov-type convergence guarantees.

    \item[C5] A four-way intersection study with nonlinear vehicle dynamics and continuous plug-and-play operations.
\end{itemize}

\section{Preliminaries and Problem Statement}

\subsection{Notation and Agent Model}
\label{subsec:model}

Let \(\mathcal I:=\{1,\ldots,M\}\) denote the set of agents. Discrete time is indexed by \(t\in\mathbb Z_+\), with \(t^+:=t+1\). For \(a,b\in\mathbb Z\), \(a\leq b\), define \(\mathbb Z_a^b:=\{k\in\mathbb Z\mid a\leq k\leq b\}\).
The Euclidean norm is denoted by \(\|\cdot\|\). For \(c\in\mathbb R^{n_p}\)
and \(R\geq 0\), let
\(
    \mathbb B(c,R):=\{p\in\mathbb R^{n_p}\mid \|p-c\|\leq R\},
\) define the Euclidean Ball. Predicted quantities use the MPC index \((k|t)\).

Each agent \(i\in\mathcal I\) is modeled by the discrete-time nonlinear
system
\begin{equation}
    x_i(t^+) = f_i(x_i(t),u_i(t)),
    \label{eq:agent_dynamics}
\end{equation}
with state \(x_i(t)\in\mathbb R^{n_i}\), input
\(u_i(t)\in\mathbb R^{m_i}\), and position
\(p_i(t)=C_i x_i(t)\in\mathbb R^{n_p}\). The state and input constraints are
given by \(\mathcal X_i\subset\mathbb R^{n_i}\) and
\(\mathcal U_i\subset\mathbb R^{m_i}\). The body of agent \(i\) is modeled
as the ball \(\mathbb B(p_i(t),r_i)\), with radius \(r_i>0\).

Each agent is assigned a reference state \(x_i^{\mathrm{ref}}\), with
reference position \(p_i^{\mathrm{ref}}=C_i x_i^{\mathrm{ref}}\). The
reference is assumed to be an admissible equilibrium, i.e., there exists
\(u_i^{\mathrm{ref}}\in\mathcal U_i\) such that
\(
    x_i^{\mathrm{ref}}=f_i(x_i^{\mathrm{ref}},u_i^{\mathrm{ref}}),
    \ x_i^{\mathrm{ref}}\in\mathcal X_i .
\)
An equilibrium of agent \(i\) is denoted by
\((\bar x_i,\bar u_i)\), and its position by
\(\bar p_i=C_i\bar x_i\). Terminal contingency equilibria are denoted by
\((\bar x_i^{\mathrm c},\bar u_i^{\mathrm c})\).

\subsection{Standing Assumptions}
\label{subsec:standing_assumptions}
\begin{assumption}[Exact dynamics]
\label{ass:model_constraints}
The closed-loop evolution of each agent is exactly described by
\eqref{eq:agent_dynamics}. No model mismatch or external disturbance is
considered.
\end{assumption}

\begin{assumption}[Position invariance]
\label{ass:position_invariance}
The dynamics are translation invariant. For each
agent \(i\), there exists an embedding
\(\phi_i:\mathbb R^{n_p}\to\mathbb R^{n_i}\) such that
\(C_i\phi_i(\delta)=\delta\) and
\[
    f_i(x_i+\phi_i(\delta),u_i)
    =
    f_i(x_i,u_i)+\phi_i(\delta)
\]
for all admissible \((x_i,u_i)\) and all \(\delta\in\mathbb R^{n_p}\).
\end{assumption}

\begin{assumption}[Finite-sensing information pattern]
\label{ass:information_pattern}
At each time \(t\in\mathbb Z_+\), agent \(i\) can measure its own state
\(x_i(t)\) and the current states \(x_j(t)\) of agents
\(j\in\mathcal N_i(t)\), where \(\mathcal N_i(t)\subseteq
\mathcal I\setminus\{i\}\) denotes its sensing neighborhood. The reference
positions, future inputs, and predicted trajectories of other agents are
not available.
\end{assumption}

\subsection{State-dependent Safe Sets}
\label{subsec:state_dependent_safe_sets}

The safety mechanism is based on local safe sets in the position space.
A safe set represents a region in which an agent admits a feasible
contingency maneuver to an admissible safe equilibrium while remaining
inside the region.

For each agent \(i\), let \(\Gamma_i\) denote a deterministic
state-dependent safe-set generator. For any \(x_i\in\mathcal X_i\),
\(\Gamma_i(x_i)\) is the safe set generated from \(x_i\). At time \(t\), we write
\begin{equation}
    \mathcal G_i(t) := \Gamma_i(x_i(t)) \subset \mathbb R^{n_p}.
    \label{eq:generated_safe_set}
\end{equation}
The set \(\mathcal G_i(t)\) is induced by the current state of agent \(i\)
and can be computed from state information alone. Its specific construction
is arbitrary as long as the requirements stated in this paper are satisfied,
in particular the existence of a feasible contingency maneuver to an
admissible safe equilibrium within \(\mathcal G_i(t)\).

The set actually used to constrain the optimal control problem (OCP) is called the active safe set and
is denoted by \(\mathcal S_i(t)\). In general, the active safe set need not
coincide with $\mathcal G_i(t)$. It may be restricted further in
order to maintain safe interaction with neighboring agents. The key
structural property imposed in this paper is
\[
    \mathcal S_i(t) \subseteq \mathcal G_i(t),
    \qquad \forall t\in\mathbb Z_+ .
\]
Thus, the active safe set is always contained in a set that can be
computed directly from the current state. This property will later enable
neighboring agents to use \(\mathcal G_i(t)\) as a conservative outer
representation of \(\mathcal S_i(t)\), without knowing the exact active
safe set or its update history.
\begin{assumption}[Local reconstructability] The safe-set generators \(\Gamma_j\), of all agents are commonly known. Hence, whenever agent
\(i\) observes the state \(x_j(t)\) of an agent \(j\in\mathcal N_i(t)\), it
can reconstruct the generated safe set
\(
    \mathcal G_j(t)=\Gamma_j(x_j(t)).
\)  
\end{assumption}

The contingency plan of agent \(i\) at time \(t\) consists of predicted
states and inputs
\begin{equation}
    X_i^{\mathrm c}(t)
    :=
    \{x^{\mathrm c}_{i,(k|t)}\}_{k=0}^{N_c},
    \quad
    U_i^{\mathrm c}(t)
    :=
    \{u^{\mathrm c}_{i,(k|t)}\}_{k=0}^{N_c-1},
    \label{ContingencyTrajectory}
\end{equation}
with predicted positions
\(
    p^{\mathrm c}_{i,(k|t)}
    =
    C_i x^{\mathrm c}_{i,(k|t)} .
\)\\
Accounting for the physical extent of agent \(i\), the contingency
trajectory is constrained by
\begin{equation}
    \mathbb B(p^{\mathrm c}_{i,(k|t)},r_i)
    \subseteq
    \mathcal S_i(t),
    \qquad
    k\in\mathbb Z_0^{N_c}.
    \label{BodyContainmentConstraint}
\end{equation}

\begin{definition}[Admissible safe equilibrium]
\label{def:admissible_safe_equilibrium}
Let \(\mathcal A\subseteq\mathbb R^{n_p}\) be a safe set. A pair
\((\bar x_i,\bar u_i)\in\mathcal X_i\times\mathcal U_i\) is called an
admissible safe equilibrium for agent \(i\) with respect to \(\mathcal A\) if
\(
    \bar x_i = f_i(\bar x_i,\bar u_i)
\)
and, with \(\bar p_i := C_i\bar x_i\),
\(
    \mathbb B(\bar p_i,r_i)\subseteq \mathcal A .
\)
\end{definition}

The contingency trajectory is required to terminate in an admissible safe
equilibrium with respect to the active safe set \(\mathcal S_i(t)\). Thus,
the terminal contingency state satisfies
\begin{equation}
    x^{\mathrm c}_{i,(N_c|t)}=\bar x_i^{\mathrm c}(t),
    \label{TerminalConstrText}
\end{equation}
where \((\bar x_i^{\mathrm c},\bar u_i^{\mathrm c})\) is an admissible safe
equilibrium with respect to \(\mathcal S_i(t)\) in the sense of
Definition~\ref{def:admissible_safe_equilibrium}. Once this equilibrium
is reached, the agent can remain inside the active safe set by applying
the constant input \(\bar u_i^{\mathrm c}\).

The role of the safe sets in multi-agent collision avoidance is
geometric. If the active safe sets of two agents are disjoint,
\[
    \mathcal S_i(t)\cap\mathcal S_j(t)=\emptyset,
    \qquad i\neq j,
\]
and both agents keep their bodies inside their respective active safe
sets, then their bodies are disjoint as well. Therefore, maintaining
disjoint active safe sets is sufficient to certify collision avoidance. How the active safe sets are updated locally while preserving pairwise disjointness is detailed in Section~\ref{sec:memory_free_safe_set_construction}.

\section{Decentralized Contingency MPC Formulation}
\label{sec:contingency_mpc_formulation}

\subsection{Nominal and Contingency Plans}
\label{subsec:nominal_and_contingency_plans}

Let \(N_n\in\mathbb Z_{>0}\) and \(N_c\in\mathbb Z_{>0}\) denote the
nominal and contingency prediction horizons, respectively. The nominal
horizon \(N_n\) is chosen according to the desired performance behavior,
while the contingency horizon \(N_c\) must be long enough to represent a
feasible fallback maneuver to an admissible safe equilibrium.

At time \(t\), agent \(i\) optimizes a nominal state-input sequence
\[
    X_i^{\mathrm n}(t)
    :=
    \{x^{\mathrm n}_{i,(k|t)}\}_{k=0}^{N_n},
    \qquad
    U_i^{\mathrm n}(t)
    :=
    \{u^{\mathrm n}_{i,(k|t)}\}_{k=0}^{N_n-1},
\]
and a contingency state-input sequence as in~\eqref{ContingencyTrajectory}.
The corresponding predicted positions are denoted by
\[
    p^{\mathrm n}_{i,(k|t)}
    :=
    C_i x^{\mathrm n}_{i,(k|t)},
    \qquad
    p^{\mathrm c}_{i,(k|t)}
    :=
    C_i x^{\mathrm c}_{i,(k|t)} .
\]
Both prediction plans are initialized at the measured state,
\[
    x^{\mathrm n}_{i,(0|t)}
    =
    x^{\mathrm c}_{i,(0|t)}
    =
    x_i(t).
\]
The nominal trajectory is optimized for the control objective of the
agent, while the contingency trajectory is constrained to remain inside
the active safe set \(\mathcal S_i(t)\) and to terminate in an admissible
safe equilibrium.
The terminal contingency equilibrium
\[
    (\bar x_i^{\mathrm c}(t),\bar u_i^{\mathrm c}(t))
    \in
    \mathcal X_i\times\mathcal U_i,
\]
is included as a decision variable
and is required to be admissible with respect to $\mathcal S_i(t)$ according
to Definition~1, with the terminal condition~\eqref{TerminalConstrText}.

The two prediction plans share the first input, which is applied to
the real system after optimization:
\begin{equation}
    u_i(t)
    =
    u^{\mathrm n,*}_{i,(0|t)}
    =
    u^{\mathrm c,*}_{i,(0|t)}.
    \label{SharedOptimalInput}
\end{equation}
This coupling ensures that the applied control action is nominally optimal while remaining contingency-feasible.

\subsection{Objective Function}
\label{subsec:objective_function}

The local objective consists of a nominal performance term and a
contingency-offset term. The nominal term is used to drive the agent
towards its reference, while the contingency-offset term favors safe
terminal equilibria close to the desired reference $x^{\mathrm{ref}}_i$.

Let \(\ell_i^{\mathrm n}(\cdot)\) denote a nominal stage cost and
\(V_i^{\mathrm n}(\cdot)\) a nominal terminal cost. Moreover, let
\(V_i^{\mathrm c}(\cdot)\) denote an offset cost for the selected terminal
contingency equilibrium. The objective of agent \(i\) at time \(t\) is
given by
\[
\begin{aligned}
J_i(t)
:=
&\sum_{k=0}^{N_n-1}
\ell_i^{\mathrm n}
\left(
x^{\mathrm n}_{i,(k|t)},
u^{\mathrm n}_{i,(k|t)}
\right)
+
V_i^{\mathrm n}
\left(
x^{\mathrm n}_{i,(N_n|t)},
x_i^{\mathrm{ref}}
\right)
\\
&+
\gamma
V_i^{\mathrm c}
\left(
\bar x_i^{\mathrm c}(t),
x_i^{\mathrm{ref}}
\right),
\end{aligned}
\]
where \(\gamma>0\) weights the preference for terminal contingency equilibria that are close to the reference. It can be shown, that for a large enough \(\gamma\) the contingency terminal state is optimal with respect to the global reference \(x_i^{\mathrm{ref}}\), see ~\cite[Lemma~1]{Saccani2023}.

\subsection{Safety and Tail-Containment Constraints}
\label{subsec:safety_tail_constraints}

By \eqref{BodyContainmentConstraint} the contingency trajectory must remain inside the active safe set
\(\mathcal S_i(t)\), including the physical radius of the agent.
This constraint certifies safety of the contingency plan with respect to
the active safe set at the current time $t$.

For recursive feasibility, however, it is not sufficient to certify only
the current contingency plan. After the first input has been applied, the
remaining tail of the contingency trajectory must still be compatible with
the safe set available at the successor state $x_i(t^+)$. Therefore, for each
prediction step \(k\), the remaining contingency tail from \(k\) onward is
required to be contained in a safe set that can be generated from the
predicted state \(x^{\mathrm c}_{i,(k|t)}\).

Let
\[
    \mathcal G_{i,(k|t)}^{\mathrm c}
    :=
    \Gamma_i(x^{\mathrm c}_{i,(k|t)})
\]
denote the safe set generated from the predicted contingency state at
stage \(k\). The tail-containment condition is
\[
    \mathbb B(p^{\mathrm c}_{i,(\ell|t)},r_i)
    \subseteq
    \mathcal G_{i,(k|t)}^{\mathrm c},
    \qquad
    k\in\mathbb Z_0^{N_c},\quad
    \ell\in\mathbb Z_k^{N_c}.
\]
This condition ensures that, after the first input is applied, the shifted
contingency tail remains contained in the state-induced safe set generated
at the successor state. It is therefore the key compatibility condition
used in the recursive-feasibility argument.

\subsection{Lyapunov-Type Constraint}
\label{subsec:lyapunov_constraint}

To obtain a Lyapunov-type convergence property, the contingency plan is
also constrained by a decreasing cost bound. Let
\(\ell_i^{\mathrm c}\) denote a nonnegative contingency stage cost. Define
the contingency cost
\[
\begin{aligned}
J_i^{\mathrm c}(t)
:=
&\sum_{k=0}^{N_c-1}
\ell_i^{\mathrm c}
\left(
x^{\mathrm c}_{i,(k|t)}-\bar x_i^{\mathrm c}(t),
u^{\mathrm c}_{i,(k|t)}-\bar u_i^{\mathrm c}(t)
\right)
\\
&+
V_i^{\mathrm c}
\left(
\bar x_i^{\mathrm c}(t),
x_i^{\mathrm{ref}}
\right).
\end{aligned}
\]
A scalar bound \(\hat J_i^{\mathrm c}(t)\in\mathbb R_+\) is maintained
recursively, and the local MPC problem imposes
\[
    J_i^{\mathrm c}(t)\leq \hat J_i^{\mathrm c}(t).
\]
After applying the shared first input, the bound is updated according to
the shifted contingency tail,
\[
\hat J_i^{\mathrm c}(t^+)
:=
J_i^{\mathrm c,*}(t)
-
\ell_i^{\mathrm c}
\left(
x_i(t)-\bar x_i^{\mathrm c,*}(t),
u_i(t)-\bar u_i^{\mathrm c,*}(t)
\right).
\]
This update is independent of the particular safe-set construction. Once
recursive feasibility of the contingency candidate has been
established, the standard shifted-tail argument yields monotone decrease
of the optimal contingency cost.

\subsection{Local Finite-Horizon Optimal Control Problem}
\label{subsec:local_fhocp}

At each time \(t\), agent \(i\) solves the following local finite-horizon
optimal control problem:
\begin{subequations}
\label{eq:local_fhocp}
\begin{align}
\min_{\substack{
X_i^{\mathrm{n,c}}(t),U_i^{\mathrm{n,c}}(t),\\
\bar x_i^{\mathrm c}(t),\bar u_i^{\mathrm c}(t)
}}
\hspace{-1.4cm}&\qquad \qquad
J_i(t)
\label{eq:local_fhocp_objective}
\\
\text{s.t.}\
&
x^{\mathrm n}_{i,(0|t)}
=
x^{\mathrm c}_{i,(0|t)}
=
x_i(t)
\label{eq:local_fhocp_initial_state}
\\
&
u^{\mathrm n}_{i,(0|t)}
=
u^{\mathrm c}_{i,(0|t)}
\label{eq:local_fhocp_shared_input}
\\
&
x^{\mathrm n}_{i,(k+1|t)}
=
f_i(
x^{\mathrm n}_{i,(k|t)},
u^{\mathrm n}_{i,(k|t)}),
\
k\in\mathbb Z_0^{N_n-1}
\label{eq:local_fhocp_nom_dyn}
\\
&
x^{\mathrm c}_{i,(k+1|t)}
=
f_i(
x^{\mathrm c}_{i,(k|t)},
u^{\mathrm c}_{i,(k|t)}),
\
k\in\mathbb Z_0^{N_c-1}
\label{eq:local_fhocp_cont_dyn}
\\
&
x^{\mathrm n}_{i,(k+1|t)}\in\mathcal X_i,\ 
u^{\mathrm n}_{i,(k|t)}\in\mathcal U_i,
\
k\in\mathbb Z_0^{N_n-1}
\label{eq:local_fhocp_nom_constraints}
\\
&
x^{\mathrm c}_{i,(k+1|t)}\in\mathcal X_i,\
u^{\mathrm c}_{i,(k|t)}\in\mathcal U_i,
\
k\in\mathbb Z_0^{N_c-1}
\label{eq:local_fhocp_cont_constraints}
\\
&
\mathbb B(p^{\mathrm c}_{i,(k|t)},r_i)
\subseteq
\mathcal S_i(t),
\
k\in\mathbb Z_0^{N_c}
\label{eq:local_fhocp_safe_set_containment}
\\
&
x^{\mathrm c}_{i,(N_c|t)}
=
\bar x_i^{\mathrm c}(t)
\label{eq:local_fhocp_terminal_state}
\\
&
\bar x_i^{\mathrm c}(t)
=
f_i\!\left(
\bar x_i^{\mathrm c}(t),
\bar u_i^{\mathrm c}(t)
\right),
(\bar x_i^{\mathrm c}(t),\bar u_i^{\mathrm c}(t))
\in
\mathcal X_i\times\mathcal U_i
\label{eq:local_fhocp_terminal_equilibrium}
\\
&
\mathbb B(\bar p_i^{\mathrm c}(t),r_i)
\subseteq
\mathcal S_i(t),
\quad
\bar p_i^{\mathrm c}(t):=C_i\bar x_i^{\mathrm c}(t)
\label{eq:local_fhocp_terminal_safe_set}
\\
&
\mathbb B(p^{\mathrm c}_{i,(\ell|t)},r_i)
\subseteq
\Gamma_i(x^{\mathrm c}_{i,(k|t)}),
\
k\in\mathbb Z_0^{N_c},\;
\ell\in\mathbb Z_k^{N_c},
\label{eq:local_fhocp_tail_containment}
\\
&
J^c_i(t) \leq \hat{J}^c_i(t).
\end{align}
\end{subequations}

\section{Safe Set Construction}
\label{sec:memory_free_safe_set_construction}

This section defines the active safe sets $\mathcal S_i(t)$ used
in~\eqref{eq:local_fhocp}. Agent $i$ first constructs its own generated
safe set $\mathcal G_i(t)$ and, for every sensed neighbor
$j\in\mathcal N_i(t)$, reconstructs $\mathcal G_j(t)$ from the measured
state $x_j(t)$.

Whenever the generated safe sets are disjoint,
\[
    \mathcal G_i(t)\cap\mathcal G_j(t)=\emptyset,
\]
agent $i$ locally constructs a pairwise separator. Let $n_{ij}(t)$
denote a separating unit direction from $\mathcal G_i(t)$ towards
$\mathcal G_j(t)$ and define
\[
    \alpha_i(t)
    :=
    \max_{p\in\mathcal G_i(t)} n_{ij}(t)^\top p,
    \qquad
    \alpha_j(t)
    :=
    \min_{p\in\mathcal G_j(t)} n_{ij}(t)^\top p .
\]
The separating boundary is placed at
\[
    b_{ij}(t)
    :=
    \frac{1}{2}
    \left(
        \alpha_i(t)+\alpha_j(t)
    \right),
\]
yielding
\[
    \mathcal H_{ij}^{\mathrm{new}}(t)
    :=
    \left\{
        p\in\mathbb R^{n_p}
        \mid
        n_{ij}(t)^\top p < b_{ij}(t)
    \right\}.
\]
Agent $j$ performs the corresponding construction locally. The stored
separator is updated whenever the generated safe sets are disjoint and
otherwise retained:
\begin{equation}
    \mathcal H_{ij}(t)
    =
    \begin{cases}
        \mathcal H_{ij}^{\mathrm{new}}(t),
        &
        \mathcal G_i(t)\cap\mathcal G_j(t)=\emptyset,
        \\[1mm]
        \mathcal H_{ij}(t-1),
        &
        \mathcal G_i(t)\cap\mathcal G_j(t)\neq\emptyset,
    \end{cases}
    \label{eq:separator_update}
\end{equation}
whenever a valid previous separator exists.

Define the currently overlapping sensed neighbors by
\[
    \mathcal N_i^{\mathrm{ov}}(t^+)
    :=
    \left\{
        j\in\mathcal N_i(t^+)
        \;\middle|\;
        \mathcal G_i(t^+)\cap\mathcal G_j(t^+)\neq\emptyset
    \right\}.
\]

Let $R_{\max}$ uniformly bound the spatial extent of the generated safe
sets. The sensing condition
\begin{equation}
    R_{\mathrm{sense}}\geq 2R_{\max}
    \label{eq:min_sensing_range}
\end{equation}
ensures that generated safe sets cannot overlap while the corresponding
agents are outside each other's sensing range.

A special case occurs if a neighbor is first sensed at $t^+$ with already
overlapping generated safe sets, i.e.,
$j\notin\mathcal N_i(t)$ and
$j\in\mathcal N_i^{\mathrm{ov}}(t^+)$. Then no valid separator is
available. By \eqref{eq:min_sensing_range},
$\mathcal G_i(t)\cap\mathcal G_j(t)=\emptyset$, and the agents initialize
the bootstrap restrictions
\[
    \mathcal B_{ij}:=\mathcal G_i(t),
    \qquad
    \mathcal B_{ji}:=\mathcal G_j(t).
\]
These restrictions are retained until the generated safe sets become
disjoint again and regular separators are initialized.

Let $\mathcal N_i^{\mathrm{sep}}(t^+)$ denote the overlapping neighbors
for which a valid stored separator exists and
$\mathcal N_i^{\mathrm{boot}}(t^+)$ those for which a bootstrap
restriction is active. Agent $i$ then locally constructs
\begin{equation}
    \mathcal S_i(t^+)
    :=
    \mathcal G_i(t^+)
    \cap
    \bigcap_{j\in\mathcal N_i^{\mathrm{sep}}(t^+)}
        \mathcal H_{ij}(t^+)
    \cap
    \bigcap_{j\in\mathcal N_i^{\mathrm{boot}}(t^+)}
        \mathcal B_{ij}.
    \label{eq:active_safe_set_update}
\end{equation}
Consequently,
\(
    \mathcal S_i(t)\subseteq\mathcal G_i(t),
    \
    \forall i\in\mathcal I,\;t\in\mathbb Z_+.
\)

Hence, $\mathcal G_i(t)$ remains a state-dependent outer representation
of the active safe set. The construction requires only current sensed
neighbor states and locally stored separators or bootstrap restrictions;
no neighbor trajectories or neighbor safe-set update histories are
required.
\section{Closed-loop Guarantees}
\label{sec:closed_loop_guarantees}

\subsection{Assumptions}

\begin{assumption}[Initial feasibility and separation]
\label{ass:initial_feasibility}
At $t=0$, the local MPC problem is feasible for every agent
$i\in\mathcal I$. Moreover,
\[
    \mathcal S_i(0)\subseteq\mathcal G_i(0),
    \qquad
    \mathcal S_i(0)\cap\mathcal S_j(0)=\emptyset,
    \quad i\neq j.
\]
\end{assumption}

\begin{assumption}[Sensing completeness]
\label{ass:sensing_completeness}
Every pair with overlapping generated safe sets is mutually detected,
i.e.,
\[
    \mathcal G_i(t)\cap\mathcal G_j(t)\neq\emptyset
    \;\Longrightarrow\;
    j\in\mathcal N_i(t),\;
    i\in\mathcal N_j(t).
\]
\end{assumption}

\subsection{Recursive Feasibility and Collision Avoidance}

\begin{remark}[Nominal completion]
\label{rem:nominal_completion}
Whenever a feasible contingency candidate exists, there also exists a
nominal candidate satisfying
\eqref{eq:local_fhocp_initial_state},
\eqref{eq:local_fhocp_shared_input},
\eqref{eq:local_fhocp_nom_dyn}, and
\eqref{eq:local_fhocp_nom_constraints}.
\end{remark}

\begin{lemma}[Safe-set containment and separation]
\label{lem:safe_set_separation}
Under Assumptions~\ref{ass:initial_feasibility} and
\ref{ass:sensing_completeness}, the active safe sets satisfy
\begin{equation}
    \mathcal S_i(t)\subseteq\mathcal G_i(t),
    \quad
    \mathcal S_i(t)\cap\mathcal S_j(t)=\emptyset,
    \quad i\neq j, \ \ \forall t\in\mathbb Z_+.
    \label{eq:safe_set_separation}
\end{equation}
\end{lemma}
\begin{proof}
Containment follows directly from
\eqref{eq:active_safe_set_update}. \\
\noindent For pairwise separation, consider arbitrary $i\neq j$. If
$\mathcal G_i(t)\cap\mathcal G_j(t)=\emptyset$, then
$\mathcal S_i(t)\cap\mathcal S_j(t)=\emptyset$ follows immediately from
$\mathcal S_i(t)\subseteq\mathcal G_i(t)$ and
$\mathcal S_j(t)\subseteq\mathcal G_j(t)$.

If a regular separator is active, then
\[
    \mathcal S_i(t)\subseteq\mathcal H_{ij}(t),
    \qquad
    \mathcal S_j(t)\subseteq\mathcal H_{ji}(t),
\]
and the separating construction gives
$\mathcal H_{ij}(t)\cap\mathcal H_{ji}(t)=\emptyset$.

Finally, if the bootstrap rule is active, then
\[
    \mathcal S_i(t)\subseteq\mathcal B_{ij},
    \qquad
    \mathcal S_j(t)\subseteq\mathcal B_{ji}.
\]
At bootstrap initialization,
$\mathcal B_{ij}=\mathcal G_i(\tau)$ and
$\mathcal B_{ji}=\mathcal G_j(\tau)$ for a time $\tau$ at which the pair
was not sensed. Hence,
$\mathcal B_{ij}\cap\mathcal B_{ji}=\emptyset$ by
\eqref{eq:min_sensing_range}. Thus the active safe sets remain pairwise
disjoint.
\end{proof}

\begin{theorem}[Recursive feasibility]
\label{thm:recursive_feasibility}
Suppose Assumptions~\ref{ass:model_constraints}, 
\ref{ass:initial_feasibility}, and \ref{ass:sensing_completeness} hold and
$\ell_i^{\mathrm c}(0,0)=0$. Then all local MPC problems are recursively
feasible, i.e., feasibility at time $t$ implies feasibility at $t^+$
for all $t\in\mathbb Z_+$.
\end{theorem}

\begin{proof}
Fix an arbitrary agent $i$ and let an optimal solution at time
$t$ be given. Since the applied input is the shared first contingency
input and the dynamics are exact,
\begin{equation}
    x_i(t^+)
    =
    x^{\mathrm c,*}_{i,(1|t)}.
    \label{eq:shift_initial_state}
\end{equation}
Consider the standard shifted contingency candidate
\begin{align}
    \tilde x^{\mathrm c}_{i,(k|t^+)}
    &=
    x^{\mathrm c,*}_{i,(k+1|t)},
    && k\in\mathbb Z_0^{N_c-1},
    \\
    \tilde x^{\mathrm c}_{i,(N_c|t^+)}
    &=
    \bar x_i^{\mathrm c,*}(t),
    \\
    \tilde u^{\mathrm c}_{i,(k|t^+)}
    &=
    u^{\mathrm c,*}_{i,(k+1|t)},
    && k\in\mathbb Z_0^{N_c-2},
    \\
    \tilde u^{\mathrm c}_{i,(N_c-1|t^+)}
    &=
    \bar u_i^{\mathrm c,*}(t),
\end{align}
with
\(
    \tilde{\bar x}_i^{\mathrm c}(t^+)
    =
    \bar x_i^{\mathrm c,*}(t),
    \
    \tilde{\bar u}_i^{\mathrm c}(t^+)
    =
    \bar u_i^{\mathrm c,*}(t).
\) \\
Together with \eqref{eq:shift_initial_state}, the contingency dynamics,
state/input constraints, and terminal-equilibrium constraints follow
directly from feasibility at time $t$.

It remains to verify the constraints affected by the safe-set update.
First, \eqref{eq:local_fhocp_tail_containment} at time $t$ with $k=1$
gives
\begin{equation}
    \mathbb B
    \left(
        \tilde p^{\mathrm c}_{i,(k|t^+)},r_i
    \right)
    \subseteq
    \Gamma_i
    \left(
        x^{\mathrm c,*}_{i,(1|t)}
    \right)
    =
    \mathcal G_i(t^+),
    \ \
    k\in\mathbb Z_0^{N_c}.
    \label{eq:shift_in_new_generated_set}
\end{equation}
Moreover, since the shifted candidate consists only of the remaining
states of the feasible contingency trajectory at time $t$, including
the repeated terminal equilibrium,
\begin{equation}
    \mathbb B
    \left(
        \tilde p^{\mathrm c}_{i,(k|t^+)},r_i
    \right)
    \subseteq
    \mathcal S_i(t),
    \qquad
    k\in\mathbb Z_0^{N_c}.
    \label{eq:shift_in_old_active_set}
\end{equation}

Every regular separator used in the construction of
$\mathcal S_i(t^+)$ contains $\mathcal S_i(t)$. Indeed, if it was
already active at time $t$, this follows directly from the previous
safe-set construction; if it is newly activated at $t^+$, it was
constructed from the disjoint generated safe sets at time $t$ and hence
contains
$\mathcal G_i(t)\supseteq\mathcal S_i(t)$.
Similarly, an already active bootstrap restriction contains
$\mathcal S_i(t)$, whereas a newly initialized one satisfies
\[
    \mathcal B_{ij}
    =
    \mathcal G_i(t)
    \supseteq
    \mathcal S_i(t).
\]
Together with
\eqref{eq:shift_in_new_generated_set}--\eqref{eq:shift_in_old_active_set}
and the update \eqref{eq:active_safe_set_update}, this yields
\[
    \mathbb B
    \left(
        \tilde p^{\mathrm c}_{i,(k|t^+)},r_i
    \right)
    \subseteq
    \mathcal S_i(t^+),
    \qquad
    k\in\mathbb Z_0^{N_c}.
\]
Hence, the active safe-set containment constraint is feasible at $t^+$;
the terminal safe-set constraint follows from $k=N_c$.

It remains to verify tail containment. For each prediction stage define
the remaining shifted contingency tail by
\[
    \tilde{\mathcal T}_{i,k}(t^+)
    :=
    \bigcup_{\ell=k}^{N_c}
    \mathbb B
    \left(
        \tilde p^{\mathrm c}_{i,(\ell|t^+)},r_i
    \right).
\]
For every $k\in\mathbb Z_0^{N_c-1}$, the shift gives
\[
    \tilde{\mathcal T}_{i,k}(t^+)
    =
    \bigcup_{\ell=k+1}^{N_c}
    \mathbb B
    \left(
        p^{\mathrm c,*}_{i,(\ell|t)},r_i
    \right),
\]
where repetition of the terminal equilibrium does not change the union.
Applying \eqref{eq:local_fhocp_tail_containment} at time $t$ with
prediction stage $k+1$ therefore yields
\[
    \tilde{\mathcal T}_{i,k}(t^+)
    \subseteq
    \Gamma_i
    \left(
        x^{\mathrm c,*}_{i,(k+1|t)}
    \right)
    =
    \Gamma_i
    \left(
        \tilde x^{\mathrm c}_{i,(k|t^+)}
    \right).
\]
For $k=N_c$, the same property follows directly from
\eqref{eq:local_fhocp_tail_containment} at the terminal stage.
Thus, the tail-containment constraint is feasible at $t^+$.

Finally, the terminal equilibrium is unchanged and the appended
equilibrium stage contributes zero stage cost. Hence,
\begin{align}
    \tilde J_i^{\mathrm c}(t^+)
    &=
    J_i^{\mathrm c,*}(t)
    -
    \ell_i^{\mathrm c}
    \left(
        x_i(t)-\bar x_i^{\mathrm c,*}(t),
        u_i(t)-\bar u_i^{\mathrm c,*}(t)
    \right)
    \nonumber\\
    &=
    \hat J_i^{\mathrm c}(t^+),
\end{align}
so the Lyapunov-bound constraint is feasible as well.

By Remark~\ref{rem:nominal_completion}, the feasible contingency candidate constructed above admits a corresponding nominal completion with the same first input. Hence, the complete local MPC problem is feasible at $t^+$. Since feasibility holds at $t=0$ by Assumption~\ref{ass:initial_feasibility}, induction establishes recursive feasibility for all $t\in\mathbb Z_+$.
\end{proof}

\begin{theorem}[Collision avoidance]
\label{thm:collision_avoidance}
Under the assumptions of
Theorem~\ref{thm:recursive_feasibility}, the closed-loop execution is
collision-free for all $t\in\mathbb Z_+$.
\end{theorem}

\begin{proof}
Recursive feasibility implies
\(
    \mathbb B(p_i(t),r_i)\subseteq\mathcal S_i(t).
\)
By Lemma~\ref{lem:safe_set_separation}, the active safe sets are
pairwise disjoint. Hence,
\[
    \mathbb B(p_i(t),r_i)
    \cap
    \mathbb B(p_j(t),r_j)
    =
    \emptyset,
    \qquad i\neq j.
\]
\end{proof}

\begin{corollary}[Plug-and-play operation]
\label{cor:pnp}
If a joining agent has a feasible local MPC problem and an initial active
safe set disjoint from all active agents, and
Assumptions~\ref{ass:sensing_completeness} remain satisfied, recursive feasibility and
collision avoidance are preserved under join and leave events.
\end{corollary}

\subsection{Lyapunov-Type Convergence}
By Theorem~\ref{thm:recursive_feasibility},
\begin{equation*}
    J_i^{\mathrm c,*}(t^+)
    \leq
    J_i^{\mathrm c,*}(t)
    -
    \ell_i^{\mathrm c}
    \left(
        x_i(t)-\bar x_i^{\mathrm c,*}(t),
        u_i(t)-\bar u_i^{\mathrm c,*}(t)
    \right).
    \label{eq:contingency_cost_decrease}
\end{equation*}
Under the standard arguments of Lyapunov stability theory
\cite{Mayne2000}, the shifted-tail
Lyapunov argument applies unchanged. The corresponding proof is given
in \cite{StudtSchildbach2026FoS,Saccani2023} and is therefore not repeated here.

\begin{figure*}[t]
    \centering
    \includegraphics[width=0.97\textwidth]{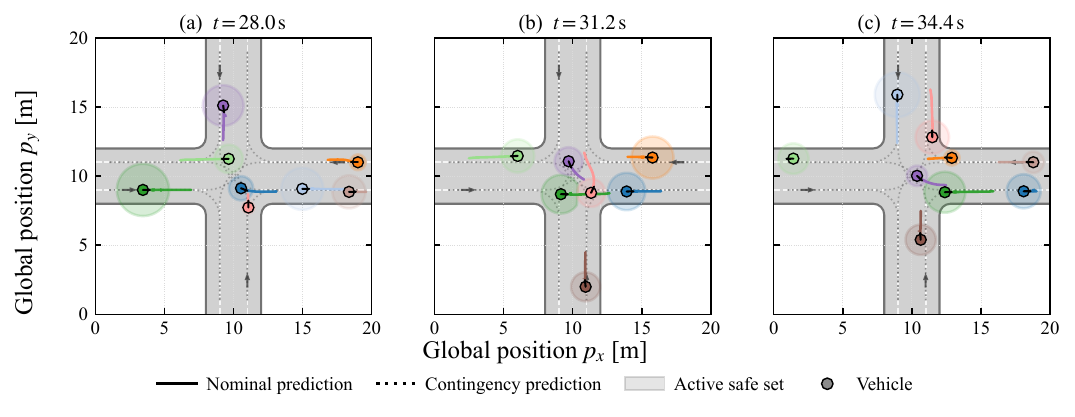}
    \caption{
        Representative snapshots of the closed-loop intersection scenario at different timesteps $t$. Solid lines denote the nominal predictions, dotted lines the contingency predictions, and shaded regions the locally active safe sets. The driving maneuver of each vehicle is unknown to all other agents.
    }
    \label{fig:intersection_snapshots}
\end{figure*}
\section{Simulation Study}
\label{sec:simulation}

The proposed framework is evaluated in an autonomous-driving scenario in which multiple vehicles traverse a four-way intersection, see Fig. \ref{fig:intersection_snapshots}. The setup incorporates nonholonomic vehicle dynamics, lane constraints, different driving maneuvers, and a continuously changing set of active agents. A demonstration video illustrating the intersection scenario is available at \url{https://youtu.be/t0YpBUVS6nU}.

\subsection{Setup}
Each vehicle is modeled by the kinematic bicycle model
\begin{equation}
    \dot p_x = v\cos(\psi), \
    \dot p_y = v\sin(\psi), \
    \dot \psi = \frac{v}{l}\tan(\delta), \
    \dot v = a,
    \label{eq:kinematic_bicycle}
\end{equation}
with state $x$ and input $u$:
\begin{equation}
    x = [p_x,p_y,\psi,v]^\top \qquad
    u = [a,\delta]^\top.
\end{equation}
Here, $p_x$ and $p_y$ denote the Cartesian position, $\psi$ the vehicle heading, $v$ the longitudinal velocity, $a$ the longitudinal acceleration, and $\delta$ the steering angle. The wheelbase is set to $l=1\,\mathrm{m}$. The continuous-time dynamics are discretized using a fourth-order Runge--Kutta scheme with a sampling time of $T_s=0.2\,\mathrm{s}$.

For collision avoidance, each vehicle is represented by a Euclidean disc of radius
$r_{\mathrm{agent}}=0.4\,\mathrm{m}$.
The state and input constraints are
\begin{equation}
    0\frac{m}{s} \leq v \leq 3\frac{m}{s}, \
    -2\frac{m}{s^2} \leq a \leq 1\frac{m}{s^2}, \ \text{and} \
    |\delta| \leq 28^\circ.
\end{equation}
Furthermore, the vehicle position is constrained to the $20\times20\,\mathrm{m}^2$ workspace. The contingency prediction is subject to hard lane-keeping constraints, whereas the nominal trajectory employs softened lane constraints to retain performance flexibility.

Vehicles enter the intersection from all four directions and may proceed straight, turn left, or turn right. New vehicles are spawned every $2\,\mathrm{s}$ whenever sufficient entrance clearance is available, while completed vehicles leave the active agent set $\mathcal{I}$. Maneuvers are sampled with probabilities $0.5$, $0.25$, and $0.25$, respectively, until $50$ vehicles have traversed the intersection.

Following the construction in Section~\ref{sec:memory_free_safe_set_construction}, the generated safe set of agent $i$ is chosen such that it contains a complete stopping maneuver after the shared first input. Let
\begin{equation}
    v_i^+(t)
    =
    \min\left\{
        v_i(t)+a_{\max}T_s,\,
        v_{\max}
    \right\}
\end{equation}
denote the largest velocity attainable after one time step. The radius of the generated safe set is then given by
\begin{equation}
    R_i(t)
    =
    r_{\mathrm{agent}}
    +
    \frac{v_i(t)+v_i^+(t)}{2}T_s
    +
    \frac{(v_i^+(t))^2}{2|a_{\min}|}.
    \label{eq:safe_set_radius}
\end{equation}
The second term bounds the displacement during the shared first input under maximum acceleration, while the third term corresponds to the subsequent braking distance. Since the Euclidean displacement from the current position cannot exceed the
traveled path length, the stopping-distance bound remains valid independently
of the steering motion. The generated safe set is therefore chosen as
\begin{equation}
    \mathcal{G}_i(t)
    =
    \mathcal{B}\!\left(p_i(t),R_i(t)\right).
\end{equation}
The active safe sets $\mathcal{S}_i(t)$ are constructed via the update rule from Section \ref{sec:memory_free_safe_set_construction}.

For the selected bounds, the maximum generated radius is
$R_{\max}=3.25\,\mathrm{m}$ and the finite sensing range is chosen as
\begin{equation}
    R_{\mathrm{sense}}
    =
    2R_{\max}
    =
    6.5\,\mathrm{m}.
\end{equation}
The nominal prediction horizon is set to $N_n=10$. The contingency horizon is selected such that one shared first input can always be followed by sufficiently many braking inputs to reach zero velocity,
\begin{equation}
    N_c
    =
    1+
    \left\lceil
        \frac{v_{\max}}
        {|a_{\min}|T_s}
    \right\rceil
    =
    9.
    \label{eq:contingency_horizon_sim}
\end{equation} 
All local nonlinear MPC problems are implemented in CasADi and solved using IPOPT \cite{IPOPT}.

\subsection{Results}

Figure~\ref{fig:intersection_snapshots} illustrates three representative time steps of the closed-loop simulation. The displayed situations highlight interactions between vehicles approaching the intersection from different directions and following mutually unknown maneuvers. The contingency predictions remain confined to the resulting admissible regions, while the nominal predictions retain sufficient flexibility to efficiently follow the assigned routes through the intersection.

The complete closed-loop behavior is summarized in Fig.~\ref{fig:intersection_trajectories}, which shows the executed trajectories of all $50$ vehicles. Despite repeated arrivals and departures and simultaneous straight, left-turn, and right-turn maneuvers, all vehicles successfully traverse the intersection and leave the workspace through their assigned exits. In particular, vehicles remain close to their prescribed lanes while locally coordinating their motion only through the state-based safe-set construction.

\begin{figure}[h!]
    \centering
    \includegraphics[width=0.8\columnwidth]{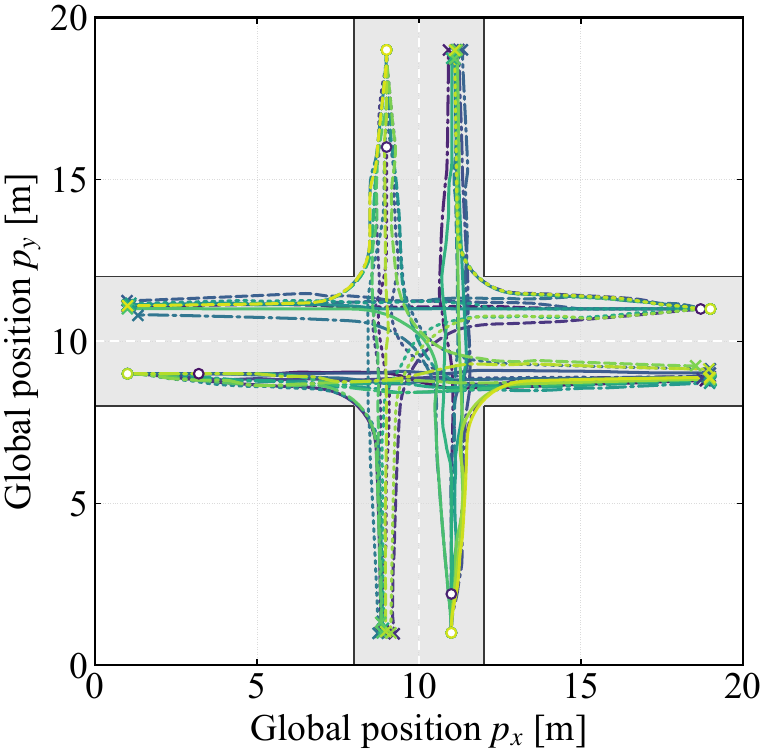}
    \caption{
        Closed-loop trajectories of all $50$ vehicles traversing the intersection.
        Each color corresponds to an individual vehicle.
        All agents successfully reach their assigned exit while remaining collision-free.
    }
    \label{fig:intersection_trajectories}
\end{figure}

The minimum center-to-center distance observed over the complete simulation is
\begin{equation}
    d_{\min} = 1.026\,\mathrm{m},
\end{equation}
which remains above the physical collision threshold
\begin{equation}
    d_{\mathrm{col}}
    =
    2r_{\mathrm{agent}}
    =
    0.8\,\mathrm{m}.
\end{equation}
Thus, no collision occurs throughout the complete experiment despite the plug-and-play operation of vehicles and the absence of communicated maneuver intentions.

Finally, the contingency costs, normalized with respect to the initial value of each individual vehicle, exhibit the Lyapunov-type behavior enforced by the contingency MPC formulation, see Fig.~\ref{fig:Lyapunov_costs}.

\begin{figure}[h!]
    \centering
    \includegraphics[width=1\columnwidth]{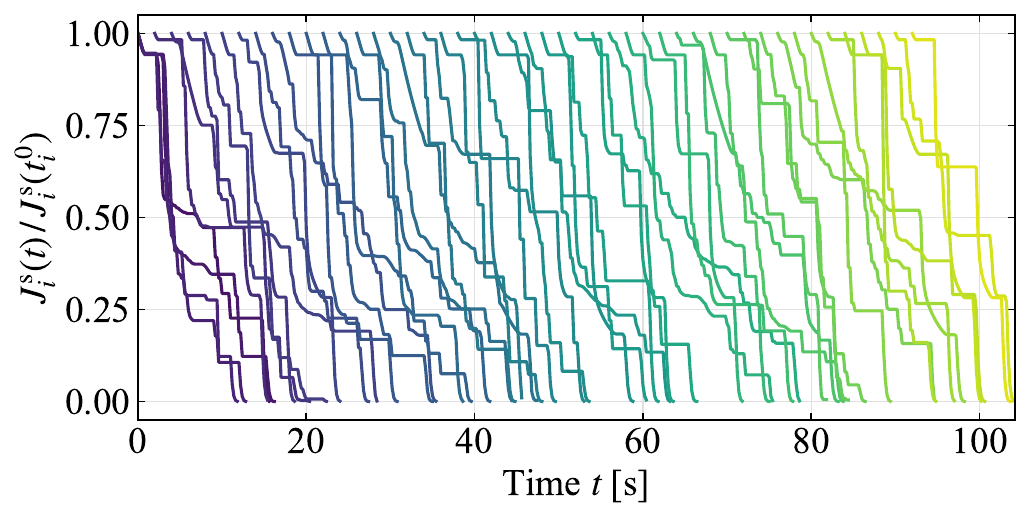}
    \caption{Closed-loop trajectories of all 50 vehicles traversing the
    intersection. Each color corresponds to an individual vehicle, and
    $t_i^0$ denotes the joining time of agent $i$. All agents successfully
    reach their assigned exit while remaining collision-free.}
    \label{fig:Lyapunov_costs}
\end{figure}

Overall, the simulation demonstrates that the proposed decentralized scheme can handle nonlinear vehicle dynamics, different and mutually unknown driving maneuvers, and dynamically changing intersection traffic while retaining a purely state-based information pattern.

\section{Conclusion and Discussion}
\label{sec:conclusion}

This paper introduced the first provable decentralized MPC under finite sensing and a state-only information pattern. By maintaining a state-dependent outer representation of the active safe sets, the proposed scheme avoids history-dependent neighbor reconstruction while preserving recursive feasibility, collision avoidance, and Lyapunov-type convergence. The intersection study demonstrates the applicability of the framework to nonlinear vehicle dynamics, unknown maneuver intentions, and plug-and-play multi-agent operation.

The present guarantees rely on exact agent models and complete detection of potentially interacting agents. Future work will therefore address robustness against model uncertainty and disturbances, as well as experimental validation and larger-scale real-time implementations.

\bibliography{refs}
\end{document}